\documentclass[11pt]{article}

\usepackage{setspace}
\include{psfig}
\usepackage{pictex, latexsym,amsmath,amssymb,amsbsy,amsfonts,amsthm,verbatim}
\usepackage[pdfpagemode=UseOutlines,colorlinks=true,pdfnewwindow=true,urlcolor=red]{hyperref}
\usepackage{color}
\usepackage{algorithm}
\usepackage{algorithmic}
\usepackage{subfigure}
\usepackage{multirow}
\usepackage[pdftex]{graphicx}
\usepackage{graphics}

\newcommand{\beq}{\begin{equation}}
\newcommand{\eeq}{\end{equation}}
\newcommand{\bea}{\begin{eqnarray*}}
\newcommand{\eea}{\end{eqnarray*}}

\input colordvi
\newcommand{\be}{\begin{equation}}
\newcommand{\ee}{\end{equation}}
\newcommand{\bx}{\begin{bmatrix}}
\newcommand{\ex}{\end{bmatrix}}

\def\eps{\varepsilon}

\theoremstyle{definition}
\newtheorem*{definition}{Definition}

\theoremstyle{theorem}
\newtheorem{theorem}{Theorem}[subsection]
\newtheorem{conjecture}[theorem]{Conjecture}

\newtheorem{corollary}[theorem]{Corollary}

\newtheorem{proposition}[theorem]{Proposition}

 \newenvironment{cem}
{
    \begin{enumerate}
        \setlength{\topsep}{0pt}
        \setlength{\parskip}{0pt}
        \setlength{\partopsep}{0pt}
        \setlength{\parsep}{0pt}         
        \setlength{\itemsep}{0pt} 
}
{
    \end{enumerate} 
}

\def\cF{\mathcal{F}}

\def\cU{\mathcal{U}}

\begin{document}

\title{Isomorph-free generation of 2-connected graphs with applications} 
\author{	Derrick Stolee\thanks{ %
	The author is supported in part by %
	the National Science Foundation grants %
	CCF-0916525 and DMS-0914815.%
	}\\ 
	Department of Mathematics\\
	Department of Computer Science\\
	University of Nebraska--Lincoln\\
	 \texttt{s-dstolee1@math.unl.edu} 
}

\maketitle

\begin{abstract}
	Many interesting graph families contain only 2-connected graphs,
		which have ear decompositions.
	We develop a technique to generate families of unlabeled 2-connected graphs 
		using ear augmentations
		and apply this technique to two problems.
	In the first application, we search for uniquely $K_r$-saturated graphs
		and find the list of uniquely $K_4$-saturated graphs
		on at most $12$ vertices, supporting current 
		conjectures for this problem.
	In the second application, we verifying the Edge Reconstruction Conjecture
			for all 2-connected graphs on at most $12$ vertices.
	This technique can be easily extended
		to more problems concerning 2-connected graphs.
\end{abstract}

\def\arbitrarypagebreak{}

\section{Introduction}

If a connected graph $G$ has a vertex $x$ so that $G - x$ is 
	disconnected or a single vertex, then $G$ is \emph{separable}.
Otherwise, $G$ is \emph{$2$-connected}, 
	and there is no single vertex whose removal
	disconnects the graph.
Many interesting graph families contain only 2-connected graphs,
	so we devise a generation technique that exploits the structure of 2-connected
	graphs.

A fundamental and well known property of 2-connected graphs
	is that they have an \emph{ear decomposition}.
An \emph{ear} is a path $x_0,x_1,\dots,x_k$
	so that $x_0$ and $x_k$ have degree at least three
	and $x_i$ has degree exactly two for all $i \in \{1,\dots,k-1\}$.
An \emph{ear augmentation} on a graph $G$
	is the addition of a path 
	with at least one edge 
	between two vertices of $G$.
The augmentation process is also invertible: an \emph{ear deletion} 
	takes an ear $x_0,x_1,\dots,x_k$ in a graph
	and deletes all vertices $x_1,\dots,x_{k-1}$
	(or the edge $x_0x_1$ if $k = 1$).
Every 2-connected graph $G$ has a sequence of subgraphs 
	$G_1 \subset \cdots \subset G_\ell = G$
	so that 
	$G_1$ is a cycle
	and for all $i \in \{1,\dots,\ell-1\}$,
	$G_{i+1}$ is the result of an ear augmentation of $G_i$~\cite{dougwest}.

In Section \ref{sec:technique}, 
	we describe a method 
	for generating 2-connected graphs
	using ear augmentations.
While we wish to generate unlabeled graphs, 
	any computer implementation must store
	an explicit labeling of the graph.
Without explicitly controlling the number of times
	an isomorphism class appears, 
	a singe unlabeled graph may appear up to 
	$n!$ times.
An \emph{isomorph-free} generation scheme 
	for a class of combinatorial objects
	visits each isomorphism class exactly once.
To achieve this goal, our strategy will make explicit use
	of isomorphisms, automorphisms, and orbits.
The technique used in this work is an implementation of
	McKay's isomorph-free generation technique~\cite{McKayIsomorphFree},
	which is sometimes called ``canonical augmentation" or ``canonical deletion".
See~\cite{ClassificationAlgorithms} for a discussion of similar techniques.
We implement this technique 
	to generate only 2-connected graphs
	using ear augmentations.
	
Almost all graphs are 2-connected~\cite{almostallkconn},
	even for graphs with a small number of vertices\footnote{
		To see the overwhelming majority of 2-connected graphs, 
		compare the number of unlabeled graphs~\cite{OEISUnlabeled} 
		to the number of unlabeled 2-connected graphs~\cite{OEIS2connected}.
	},
	so as a method of generating all 2-connected graphs,
	this method cannot significantly reduce computation 
	compared to 
	generating all graphs and ignoring the separable graphs.
The strength of the method lies in its application
	to search over ear-monotone properties
	and to use the structure of the search to reduce computation.
These strengths are emphasized in two applications
	presented in this work.

In Section \ref{sec:saturation},
	we search for graphs that are uniquely $K_r$-saturated.
These graphs contain no $K_r$ and adding any edge from the complement
	creates a unique copy of $K_r$.
This pair of constraints 
	reduces the number of graphs that are visited while
	searching for uniquely $K_r$-saturated graphs.
The graphs found with this method support the current conjectures
	on these graphs.

In Section \ref{sec:reconstruction},
	we verify the Edge Reconstruction Conjecture on 
	small 2-connected graphs.
The structure of the search allows for a reduced number of pairwise comparisons 
	between edge decks.
Also, it is known that the Reconstruction Conjecture holds if
	all 2-connected graphs are reconstructible.
Since graphs with more than $1+\log(n!)$ edges are edge-reconstructible,
	we focus only on 2-connected graphs with at most this number of edges,
	providing a sparse set of graphs to examine.
This verifies the conjecture on all 2-connected graphs
	up to 12 vertices, 
	extending previous results~\cite{McKaySmallReconstruction}.
	
\subsection{Notation}

In this work, $H$ and $G$ are graphs, 
	all of which will are simple: there are no loops or multi-edges.
For a graph $G$, $V(G)$ is the vertex set and $E(G)$ is the edge set.
The number of vertices is denoted $n(G)$, while $e(G)$ is the number of edges.
The complement graph $\overline{G}$ is the graph on vertices $V(G)$ 
	with a vertex pair $xy$ in $E(\overline{G})$
	if and only if $xy \notin E(G)$.

For a 2-connected graph, a vertex of degree at least three is a 
	\emph{branch vertex}.
Vertices of degree two are \emph{internal} vertices, as they are
	contained between the endpoints of an ear.
Ears will be denoted with $\eps$.
For an ear $\eps$, the \emph{length} of $\eps$ is the number of edges between
	the endpoints
	and its \emph{order} is the number of internal vertices between the endpoints.
We will focus on the order of an ear.
An ear of order $0$ (length $1$) is a single edge, 
	called a \emph{trivial} ear.
Ears of larger order are \emph{non-trivial}.

Given a graph $G$ and an ear $\eps = x_0,x_1,\dots,x_k$, 
	the \emph{ear deletion} $G-\eps$ is the graph $G-x_1-x_2-\cdots-x_{k-1}$,
	where all internal vertices of $\eps$ are removed.
For an ear $\eps = x_0,x_1,\dots,x_{k-1},x_k$ where
	$x_0, x_k \in V(G)$ but
	$x_1,x_2,\dots,x_{k-1}$ are not vertices in $G$,
	the \emph{ear augmentation} $G+\eps$
	is given by adding the internal vertices of $\eps$ to $G$
	and adding the edges $x_ix_{i+1}$ for $i \in \{0,\dots,k-1\}$.

\section{Isomorph-Free Generation via Ear Decompositions}
\label{sec:technique}

In this section, 
	we describe a general method for performing isomorph-free generation
	in specific families of 2-connected graphs.
	
\subsection{The search space and ear augmentation}

	Consider a family $\cF$ of unlabeled 2-connected graphs.
	We say $\cF$ is \emph{deletion-closed}
		if every graph $G$ in $\cF$ which is not a cycle
		has an ear $\eps$ so that the ear deletion $G - \eps$ is also in $\cF$.
	For an integer $N \geq 3$, 
		$\cF_{ N}$ is the set of graphs in $\cF$ 
		with at most $N$ vertices.
	

This requirement implies that 
	for every graph $G \in \cF$, 
	there exists a sequence $G \supset G_1 \supset G_2 \cdots$
	of ear deletions $G_{i+1} = G_i - \eps_i$ 
	where each graph $G_i$ is in $\cF$ 
	and the sequence $\{G_i\}$ terminates at some cycle $C_k \in \cF$.
By selecting an ear deletion which is invariant 
	to the representation of each $G_i$, 
	we define a canonical sequence of ear-deletions 
	that terminates at such a cycle.
While generating graphs of $\cF$, we
	shall only follow augmentations that correspond to these 
	canonical deletions, giving 
	a single sequence of augmentations for each isomorphism class in $\cF$.
This allows us to visit each isomorphism class in $\cF$
	exactly once using a backtracking search and
	without storing a list of previously visited graphs.
	
The search structure is that of a rooted tree: the 
	root node is an empty graph, with 
	the first level of the tree given by each cycle $C_k$ in $\cF_N$.
Each subsequent search node is extended upwards by 
	all canonical ear augmentations.
Since the search does not require a list of previously visited graphs,
	disjoint subtrees are independent and can be run concurrently
	without communication.
This leads to a search method which can be massively parallelized
	without a significant increase in overhead.

Note that being deletion-closed does not imply that \emph{every}
	ear $\eps$ in $G$ has $G-\eps$ in the family.
In fact, this does not even hold for the family of 2-connected graphs,
	as removing some ears leave the graph separable.
See Figure \ref{fig:EarSeparable} for an example of such an ear deletion.

\begin{figure}[t]
	\centering
	\mbox{
		\subfigure[A 2-connected graph $G$ and an ear $\eps$.]{
			\hspace{0.4in}
			\includegraphics[width=0.2\textwidth]{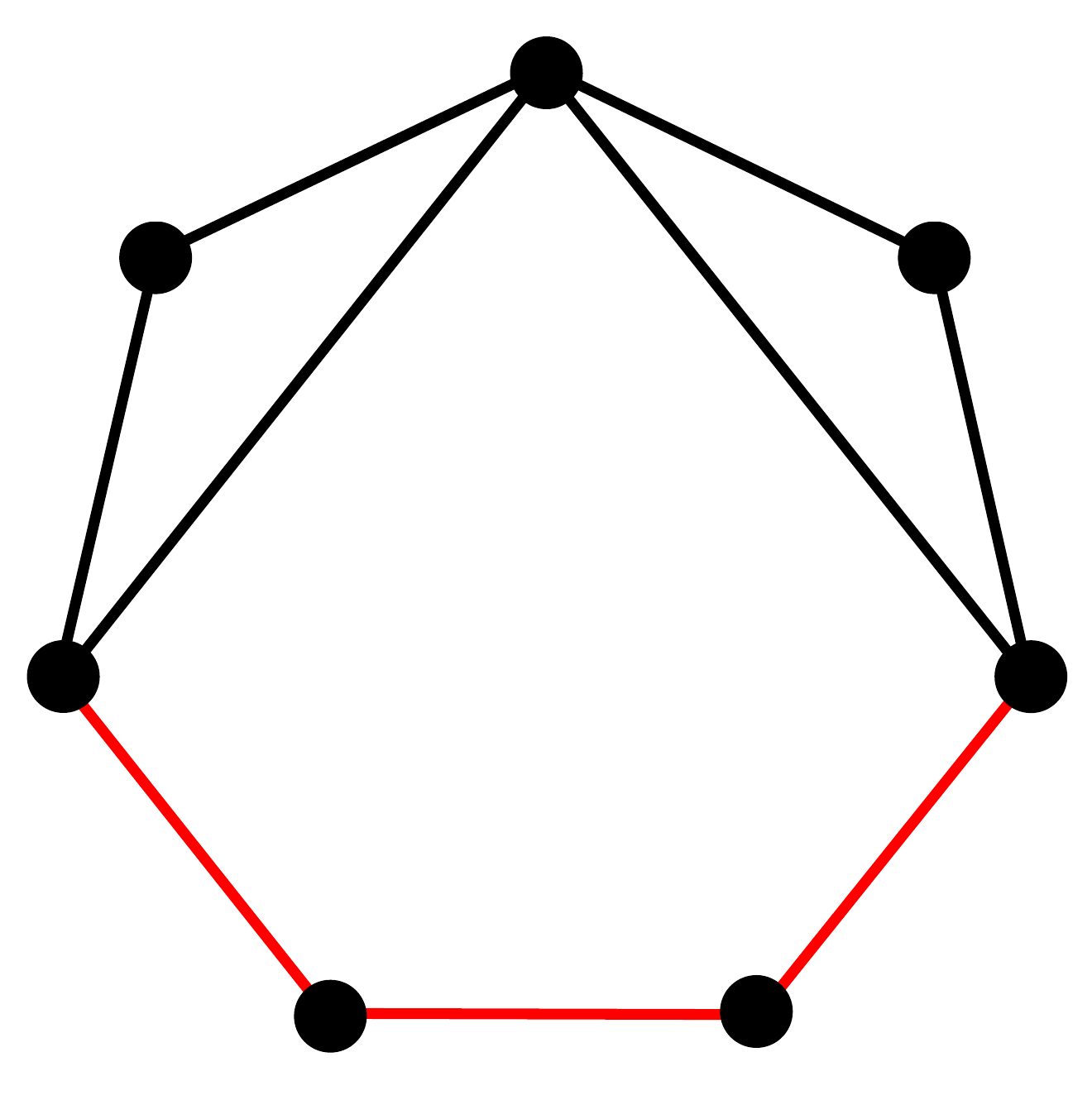}
			\hspace{0.4in}
		}
		\qquad
		\subfigure[$G-\eps$, separable.]{
			\hspace{0.4in}
			\includegraphics[width=0.2\textwidth]{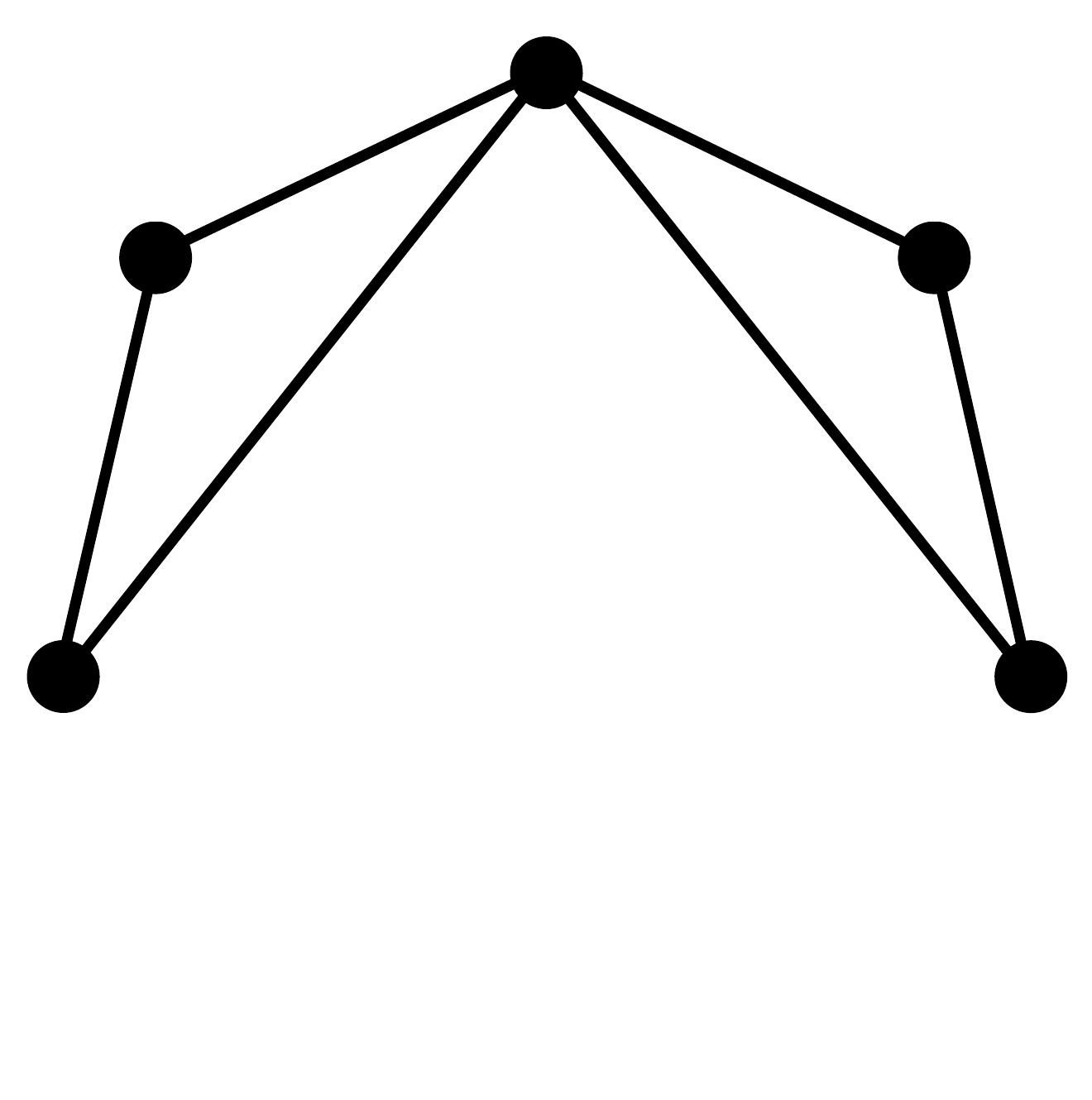}
			\hspace{0.4in}
		}
	}
	\caption{\label{fig:EarSeparable}A 2-connected graph $G$ and an ear $\eps$ %
		 whose removal makes $G-\eps$ separable.}
\end{figure}

Also, if $\cF$ is deletion-closed, then so is $\cF_{ N}$.
While the algorithms described could operate over $\cF$,
	a specific implementation will have a bounded number ($N$) 
	of vertices to consider.
Operating over $\cF_{ N}$ allows 
	for a finite number of possible ear augmentations at each step.

To augment a given labeled graph $G$,
	enumerate all pairs of vertices $x, y \in V(G)$
	and orders $r \geq 0$ so that $|V(G)|+r \leq N$
	and attempt adding an ear between $x$ and $y$
	of order $r$.
If an edge exists between $x$ and $y$, 
	then adding an ear of order $0$
	will immediately fail.
However, all other orders produce valid 2-connected graphs.
We then test if the augmentation $G+\eps$ is in $\cF$,
	discarding graphs which are not in the family.

\arbitrarypagebreak
\subsection{Augmenting by orbits}

By considering the automorphisms of a given graph,
	we can reduce the number of attempted ear augmentations.
First, note that between a given pair of vertices,
	multiple ears of the same order are in orbit with each other.
Second, if $\eps_1$ is an ear between $x_1$ and $y_1$
	and $\eps_2$ is an ear between $x_2$ and $y_2$, 
	then $\eps_1$ and $\eps_2$ are in orbit 
	if and only if they have the same order
		and the vertex sets $\{x_1,y_1\}$, $\{x_2,y_2\}$
		are in orbit under the automorphism group of $G$.
Third, if the sets of vertices $\{x_1,y_1\}$ and 
	$\{x_2,y_2\}$ are in orbit under the automorphism group of $G$,
	then the augmentations formed by 
	adding an ear of order $r$
	between $x_1$ and $y_1$ is
	isomorphic to adding an ear of order $r$
	between $x_2$ and $y_2$.

This redundancy under graphs with non-trivial automorphism group
	is removed by computing the orbits of vertex pairs,
	then only augmenting ears between
	a single representative of a pair orbit.
Pair orbits are computed by applying the
	generators of the automorphism group of $G$
	to the set of vertex pairs. 


\subsection{Canonical deletion of ears}
\def\del{\text{Delete$_{\cF}$}}

While augmenting by orbits reduces the 
	number of generated graphs,
	a \emph{canonical deletion} is defined
	to guarantee that each unlabeled graph in $\cF_{ N}$
	is enumerated exactly once.
This selects a unique ear $\eps = \del(G)$ 
	so that $G-\eps$ is in $\cF$
	and $\eps$ is invariant to the labeling of $G$.
That is, if $G_1$ and $G_2$ are isomorphic graphs
	with deletions 
	$\del(G_1) =  \eps_1$
	and $\del(G_2) = \eps_2$, 
	then there is an isomorphism
	$\pi$ from $G_1$ to $G_2$
	so that $\pi$ maps $\eps_1$ to $\eps_2$.
	
\def\lab{\operatorname{lab}}
In order to compute a representative $\del(G)$ 
	that is invariant to the 
	labels of $G$, 
	a canonical labeling of $V(G)$
	is computed.
A \emph{canonical labeling} is a map $\lab(G)$
	which maps graphs $G$ to 
	permutations $\pi_G : V(G) \to \{0,1,2,\dots,|V(G)|-1\}$
	so that for every labeled graph $G' \cong G$,
	the map $\phi : V(G) \to V(G')$ given by 
	$\phi(v) = \pi_{G'}^{-1}(\pi_{G}(v))$ for each $v \in V(G)$ 
	is an isomorphism from $G$ to $G'$.
In this sense, the map $\pi_G$ is invariant to the labels of $V(G)$.
McKay's \texttt{nauty} library~\cite{nauty,HRnauty} is used to 
	compute this canonical labeling.

Once the canonical labeling is computed, 
	the canonical deletion can be chosen  
	by considering all ears $\eps$ whose deletion ($G-\eps$)
	remains in $\cF_{ N}$,
	and selecting the ear with 
	(a) minimum length, and
	(b) lexicographically-least canonical label of branch vertices.
Algorithm \ref{alg:CanonDelete}
	details this selection procedure.

\begin{algorithm}[t]
	\caption{\label{alg:CanonDelete}Delete$_{\cF}G)$ --- %
		The Default Canonical Deletion in $\cF$}
	\begin{algorithmic}
		\STATE minOrder $\leftarrow n(G)$
		\STATE minLabel $\leftarrow n(G)^2$
		\STATE bestEear $\leftarrow$ {\bf null}
		\FOR{all vertices $x \in V(G)$ with $\deg x \geq 3$}
			\FOR{all ears $e$ incident to $x$}
				\STATE Let $y$ be the opposite endpoint of $e$
				\STATE label $\leftarrow \min\{ n(G)\pi_{G}(x) + \pi_{G}(y), n(G)\pi_{G}(y) + \pi_{G}(x)\}$
				\STATE $r \leftarrow$ order of $e$
				\IF{$G-e \in \cF_{ N}$}
					\IF{$r < $ minOrder}
						\STATE minOrder $\leftarrow r$
						\STATE minLabel $\leftarrow $ label
						\STATE bestEar $\leftarrow (x,y,r)$
					\ELSIF{$r = $ minOrder {\bf and} label $<$ minLabel}
						\STATE minLabel $\leftarrow$ label
						\STATE bestEar $\leftarrow (x,y,r)$
					\ENDIF
				\ENDIF
			\ENDFOR
		\ENDFOR
		\RETURN bestEar
	\end{algorithmic}
\end{algorithm}


\subsection{Full implementation}

This isomorph-free generation scheme is formalized by
	the recursive algorithm Search$_{\cF}$($G, N$), 
	given in Algorithm \ref{alg:FullSearch}.
The full algorithm Search$_{\cF}$($N$) 
	searches over all graphs of order at most $N$ in $\cF$
	and
	is 
	initialized by calling Search$_{\cF}$($C_k, N$)
	for each $k \in \{3,4,\dots,N\}$.
Since the recursive calls to Search$_{\cF}(G,N)$
	are independent, they can be run concurrently
	without communication.

\begin{algorithm}[t]
	\caption{\label{alg:FullSearch} Search$_{\cF}$($G,N$) --- %
		Search all canonical augmentations of $G$ in $\cF_{N}$}
	\begin{algorithmic}
		\IF{Prune$_{\cF}$($G$) $=$ {\bf true}}
			\RETURN
		\ENDIF
		\IF{$G$ is a solution}
			\STATE Store $G$
		\ENDIF
		\STATE $R \leftarrow N - n(G)$ 
		\FOR{all vertex-pair orbits $\mathcal{O}$}
			\STATE $\{x,y\} \leftarrow$ representative pair of $\mathcal{O}$
			\FOR{all orders $r \in \{0,1,\dots, R\}$}
				\STATE $G' \leftarrow G + \operatorname{Ear}(x,y,r)$
				\STATE $(x',y',r') \leftarrow$ Delete$_{\cF}(G')$
				\IF{$r = r'$ {\bf and} $\{x',y'\} \in \mathcal{O}$}
					\STATE Search$_{\cF}$($G',N$)
				\ENDIF
			\ENDFOR
		\ENDFOR
		\RETURN
	\end{algorithmic}
\end{algorithm}

For some applications, 
	it is possible
	to determine that  
	no solutions are reachable 
	under any sequence of ear augmentations.
In such a case, the algorithm can 
	stop searching at the current node
	to avoid computing all augmentations
	and canonical deletions.
Let Prune$_{\cF}(G)$ be the subroutine
	which detects if such a pruning is possible.

The framework for Algorithm \ref{alg:FullSearch}
	was implemented in the TreeSearch library\footnote{The TreeSearch library is available at \href{https://github.com/derrickstolee/TreeSearch}{https://github.com/derrickstolee/TreeSearch}}~\cite{TreeSearch},
	a C++ library for managing 
	a distributed search using the Condor scheduler~\cite{condor-practice}.
This implementation was executed
	on the Open Science Grid~\cite{OpenScienceGrid}
	using the University of Nebraska Campus Grid~\cite{WeitzelMS}.
Performance calculations in this paper are based on 
	the accumulated CPU time over this heterogeneous set of computation servers.
For example, the nodes available on the University of Nebraska Campus Grid
	consist of Xeon and Opteron processors with 
	a speed range of 2.0-2.8 GHz.
All code and documentation written for this paper
	are available in a GitHub repository\footnote{The EarSearch library is available at \href{https://github.com/derrickstolee/EarSearch}{https://github.com/derrickstolee/EarSearch}}.
	
\subsection{Generating all 2-connected graphs}

Using the isomorph-free generation scheme of canonical ear deletions,
	we can generate all unlabeled 2-connected graphs on $N$ vertices
	or graphs on $N$ vertices with exactly $E$ edges.

\begin{definition}
	Let $N$ and $E$ be integers.  
	Set $g_N$ to be the number of 
		unlabeled 2-connected graphs on $N$ vertices
		and $g_{N,E}$ to be the number
		of unlabeled 2-connected graphs on 
		$N$ vertices and $E$ edges.
	${\mathcal G}_N$ is the family of 2-connected graphs
		on up to $N$ vertices.
	${\mathcal G}_{N,E}$ is the family of 2-connected graphs
		on up to $N$ vertices and up to $E$ edges.
\end{definition}

Robinson~\cite{robinson1970enumeration}
	computed the values of $g_N$  and $g_{N,E}$, 
	listed in~\cite{OEIS2connected,RobinsonTables}.
Note that ${\mathcal G}_N$ and ${\mathcal G}_{N,E}$ are 
	deletion-closed families,
	and can be searched using isomorph-free generation
	via ear augmentations.
We revisit the three main behaviors of the algorithm:
	canonical deletion, pruning, and determining solutions.

\noindent{\bf Canonical Deletion:}
	The canonical deletion algorithm in Algorithm \ref{alg:CanonDelete}
	suffices for the class of 2-connected graphs.
	Recall this algorithm selects from ears $\eps$ so that $G-\eps$ 
		remains 2-connected, selecting one of minimum length
		and breaking ties by using the canonical labels 
		of the endpoints.

\noindent{\bf Pruning:} 
If the number of edges is fixed to be $E$, 
	a graph with more than $E$ edges should be pruned.
	Also, a graph on $n(G) < N$ vertices must add at least 
		$N-n(G)+1$ edges during ear augmentations 
		in order to achieve $N$ total vertices.
	If $e(G) + (N-n(G)+1) > E$, then
		no graph on $N$ vertices with at most $E$ edges
		can be reached by ear augmentations from $G$.
	In this case, the node can be pruned.

\noindent{\bf Solutions:}
	 A 2-connected graph is a solution if and only if
	 $n(G) = N$, and if $E$ is specified then $e(G) = E$ must also hold.

\begin{table}[t]
	\centering
	\begin{tabular}[h]{c|r|r@{ }r@{ }r@{ }r@{.}l}
		\multicolumn{1}{c|}{$N$} & \multicolumn{1}{c|}{$g_N$} & \multicolumn{5}{c}{CPU time}\\
		\hline 
		  5 &        10 & &&& 0&01s \\
		  6 &        56 & &&& 0&11s \\
		  7 &       468 & &&& 0&26s \\
		  8 &      7123 & &&& 10&15s \\
		  9 &    194066 & && 5m & 17&27s \\
		 10 &   9743542 & &  7h & 39m  & 28&47s \\
		 11 & 900969091 & 71d & 22h & 22m  & 49&12s \\
	\end{tabular}
	\caption{\label{tab:2connectedTime}Comparing $g_N$ and %
		the time to generate $\mathcal{G}_{N}$.}
\end{table}

Table \ref{tab:2connectedTime} compares the number of 2-connected graphs
	of order $N$ and
	the CPU time to enumerate all such graphs.
Both the computation times and the sizes of the sets grow exponentially.
Since the number of 2-connected graphs on $N$ vertices grows so quickly,
	to test the performance for larger orders,
	the number of edges was also fixed to be slightly more than $N$.
Table \ref{tab:CNE} shows these computation times.

\begin{table}[t]
	\centering
	\small
	\begin{tabular}[h]{r||r|r|r|r|r|r|r|r|r|r}
		$N$ &  $E=11$ & $E=12$ & $E=13$ & $E=14$ & $E=15$ 
		 & $E=16$ & $E=17$ & $E=18$ & $E=19$ & $E=20$ \\
		\hline 
		 \multirow{2}{*}{$10$} & 9 & 121 & 1034 & 5898 & 23370 & 69169 & 162593 & 317364 & 530308 & 774876 \\
				 			    & 0.01 &  0.16 & 1.73 &12.99 & 65.88 & 167.12 & 472.68 & 972.62 & 2048.85 & 3631.71 \\
		\hline
		 \multirow{2}{*}{$11$} && 11 & 189  & 2242 & 17491 & 94484 & 380528 & 1212002 & 3194294 & 7197026  \\
				 			   && 0.02 & 0.38 & 5.52 & 56.10 & 260.53 & 1212.89 & 4069.09 & 13104.24 & 32836.53 \\
		\hline
		 \multirow{2}{*}{$12$}  &&& 13 & 292 & 4544 & 46604 & 334005 & 1747793 & 7274750 &
 24972998 \\
				 			   &&& 0.03 & 0.86 & 17.56 & 286.00 & 1226.71 & 6930.00 & 33066.80 & 125716.68  \\
		\hline
		 \multirow{2}{*}{$13$}  &&&& 15 & 428 & 8618 & 113597 & 1031961 & 6945703 & 36734003 \\
				 			    &&&& 0.05 & 1.83 & 44.64 & 469.02 & 5174.92 & 39018.15 & 227436.84 \\
		\hline
		 \multirow{2}{*}{$14$}  &&&&& 18 & 616 & 15588 & 257656 & 2925098 & 24532478  \\
				 			    &&&&& 0.08 & 3.82 & 90.51 & 1573.81 & 21402.18 & 183482.70 \\
		\hline
		 \multirow{2}{*}{$15$}  &&&&&& 20 & 855 & 26967 & 519306 & 7654299  \\
				 			     &&&&&& 0.12 & 7.56 & 198.84 & 4567.43 & 76728.79 \\
		\hline
		 \multirow{2}{*}{$16$}  &&&&&&& 23 & 1176 & 44992 & 1111684  \\
				 			     &&&&&& & 0.18 & 15.56 & 498.20 & 13176.05 \\
	\end{tabular}
	\caption{\label{tab:CNE}Comparing $g_{N,E}$ (above) and the time to generate $\mathcal{G}_{N,E}$ (below, in seconds).}
\end{table}

\section{Application 1: Uniquely $H$-Saturated Graphs}
\label{sec:saturation}

Our first application
	forbids certain subgraphs,
	which decreases  
	the number of graphs to enumerate.
We investigate \emph{uniquely $H$-saturated} graphs.
	
\begin{definition}
	Let $H$ and $G$ be graphs.
	$G$ is \emph{$H$-saturated}
		if $G$ contains no copy of $H$
		and for every edge $e \in E(\overline{G})$ 
		there is at least one copy of $H$ in $G+e$.
	$G$ is \emph{uniquely $H$-saturated}
		if $G$ contains no copy of $H$
		and for every edge $e \in E(\overline{G})$,
		there is a unique copy of $H$ in $G + e$.
\end{definition}

While it is easy to see that $H$-saturated graphs always exist,
	being uniquely $H$-saturated is a very strict condition.
In fact, not all $H$ admit \emph{any} graph which is uniquely $H$-saturated.
For $k \in \{6,7,8\}$, 
	no uniquely $C_k$-saturated graphs exist~\cite{WengerDiss}.
For other graphs $H$, there is a very limited list of uniquely $H$-saturated graphs.
If $G$ is uniquely $C_3$-saturated, 
	then $G$ is either a star ($K_{1,n}$) or 
	a Moore graph of diameter two and girth five:
	$G$ has no triangles and every pair of non-adjacent vertices
	have exactly one common neighbor.
There are at least three Moore graphs: 
	$C_5$, the Petersen graph, the Hoffman-Singleton graph,
	and possibly some $57$-regular graphs on $3250$ vertices~\cite{MooreGraph}.
There are exactly ten uniquely $C_4$-saturated graphs~\cite{UniqueC4Saturation}.
If $G$ is uniquely $C_5$-saturated, then $G$ is either a friendship graph 
	(every pair of vertices have exactly one common neighbor) or one of
	a finite number of other examples~\cite{WengerDiss}.
The only friendship graphs are the \emph{windmills}: 
	$\frac{n-1}{2}$ triangles sharing a common vertex~\cite{FriendshipGraphs}.

\subsection{Uniquely $K_r$-saturated graphs}
	
Historically, the first host graph $H$ where
	the extremal problems on $H$-saturated graphs
	were solved was the complete graph $K_r$~\cite{Turan,KrSaturation}.
However, uniquely $K_r$-saturated graphs have
	evaded attempts at classification.
Only empty graphs are uniquely $K_2$-saturated,
	and uniquely $K_3$-saturated graphs are stars and Moore graphs 
	(since $K_3 \cong C_3$).
There are two known infinite families of uniquely $K_r$-saturated graphs:
	books and cycle complements.

The $t$-book on $n$ vertices 
	is a complete graph $K_t$ (the \emph{spine}) 
	joined with an independent set
	on $n-t$ vertices (the \emph{pages}).
The $(r-2)$-book has cliques of size at most $r-1$ 
	and all non-edges are in the independent set.
Adding any edge in the independent set forms exactly one $K_r$
	by using the two endpoints and the $r-2$ vertices in the spine.
Figures \ref{subfig:uk3-b}, \ref{subfig:uk4-b}, and \ref{subfig:uk5-b} 
	are examples of $(r-2)$-books for $r \in \{3,4,5\}$.
For $r = 3$, note that the $(r-2)$-book 
	with $n$ pages is isomorphic to the star $K_{1,n}$ 
	with $n$ leaves.

\def\ukfigwidth{0.13\textwidth}
\begin{figure}[t]
	\centering
	\begin{tabular}[h]{cccccc}
		\subfigure[\label{subfig:uk3-b}1-book]{\includegraphics[width=\ukfigwidth]{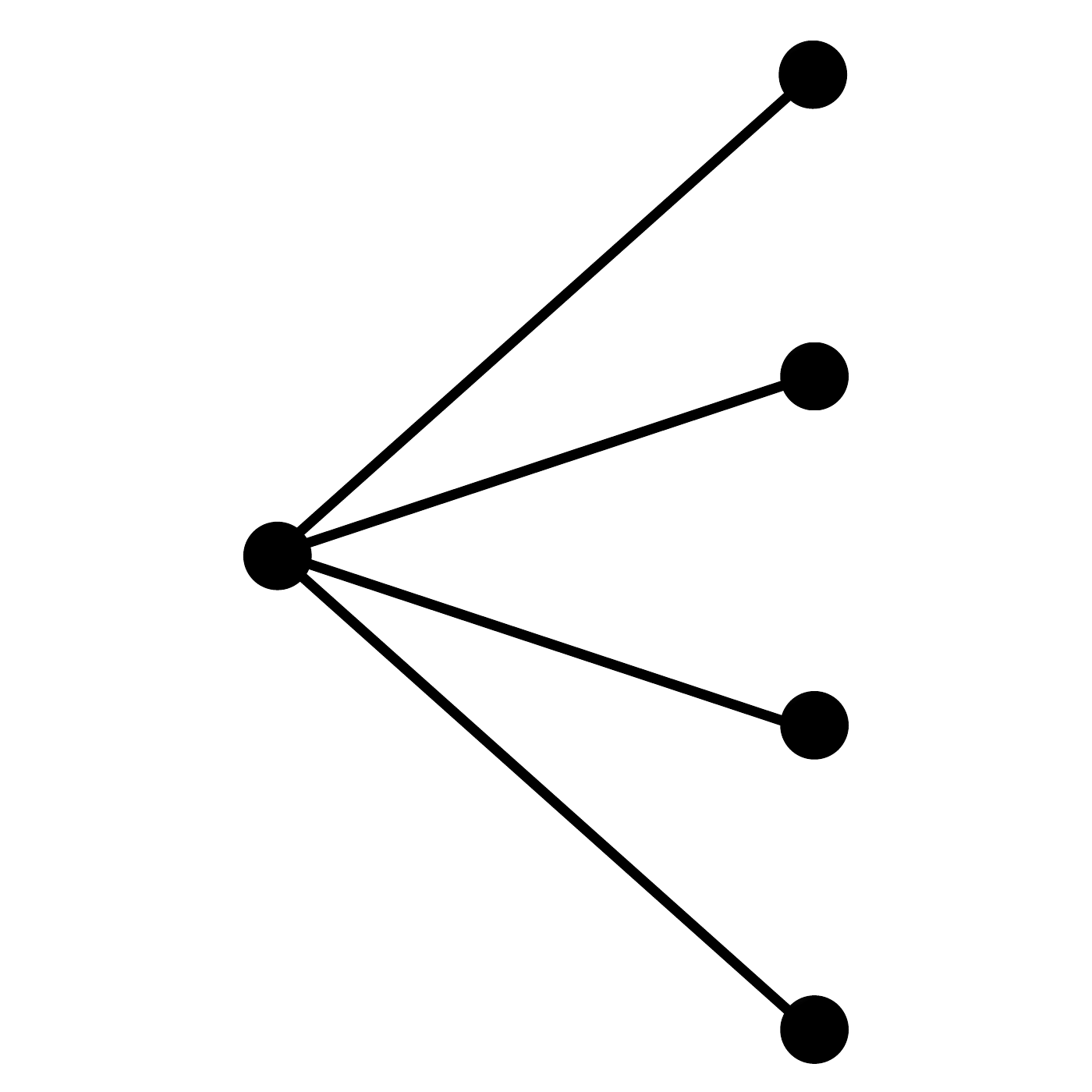}}
		&
		\subfigure[\label{subfig:uk4-b}2-book]{\includegraphics[width=\ukfigwidth]{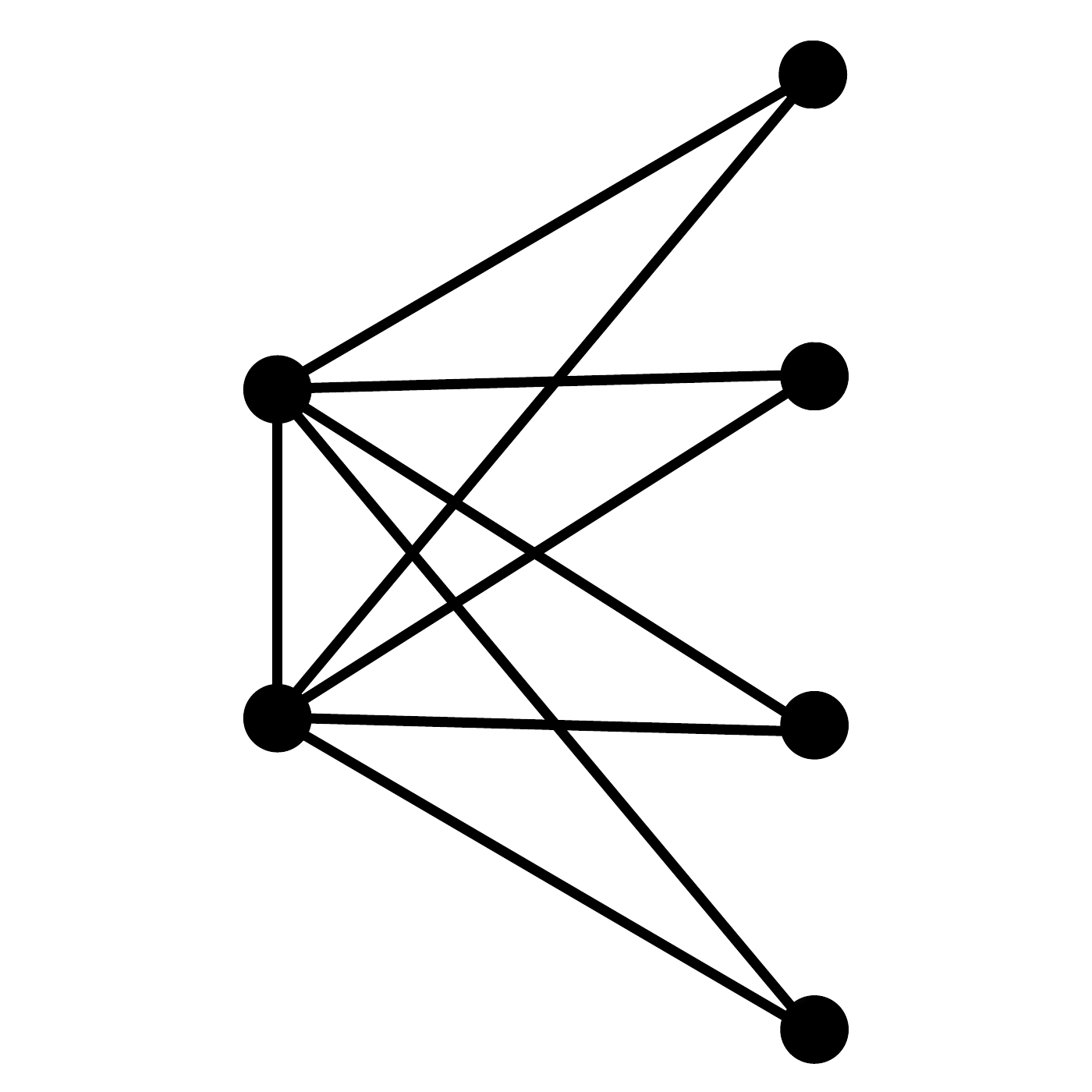}}
		&
		\subfigure[\label{subfig:uk5-b}3-book]{\includegraphics[width=\ukfigwidth]{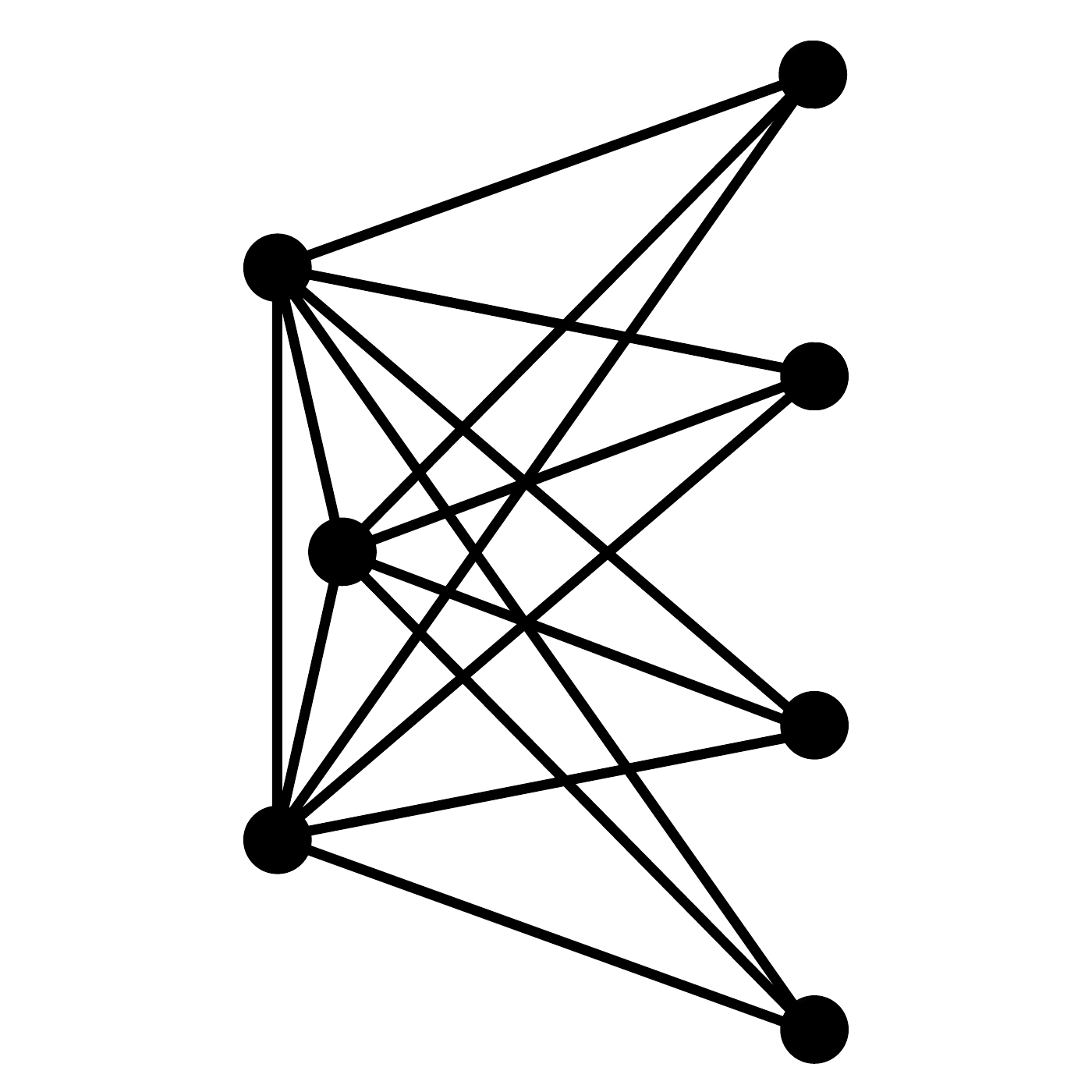}}
		&
		\subfigure[\label{subfig:cycle-5}$\overline{C}_5$]{\includegraphics[width=\ukfigwidth]{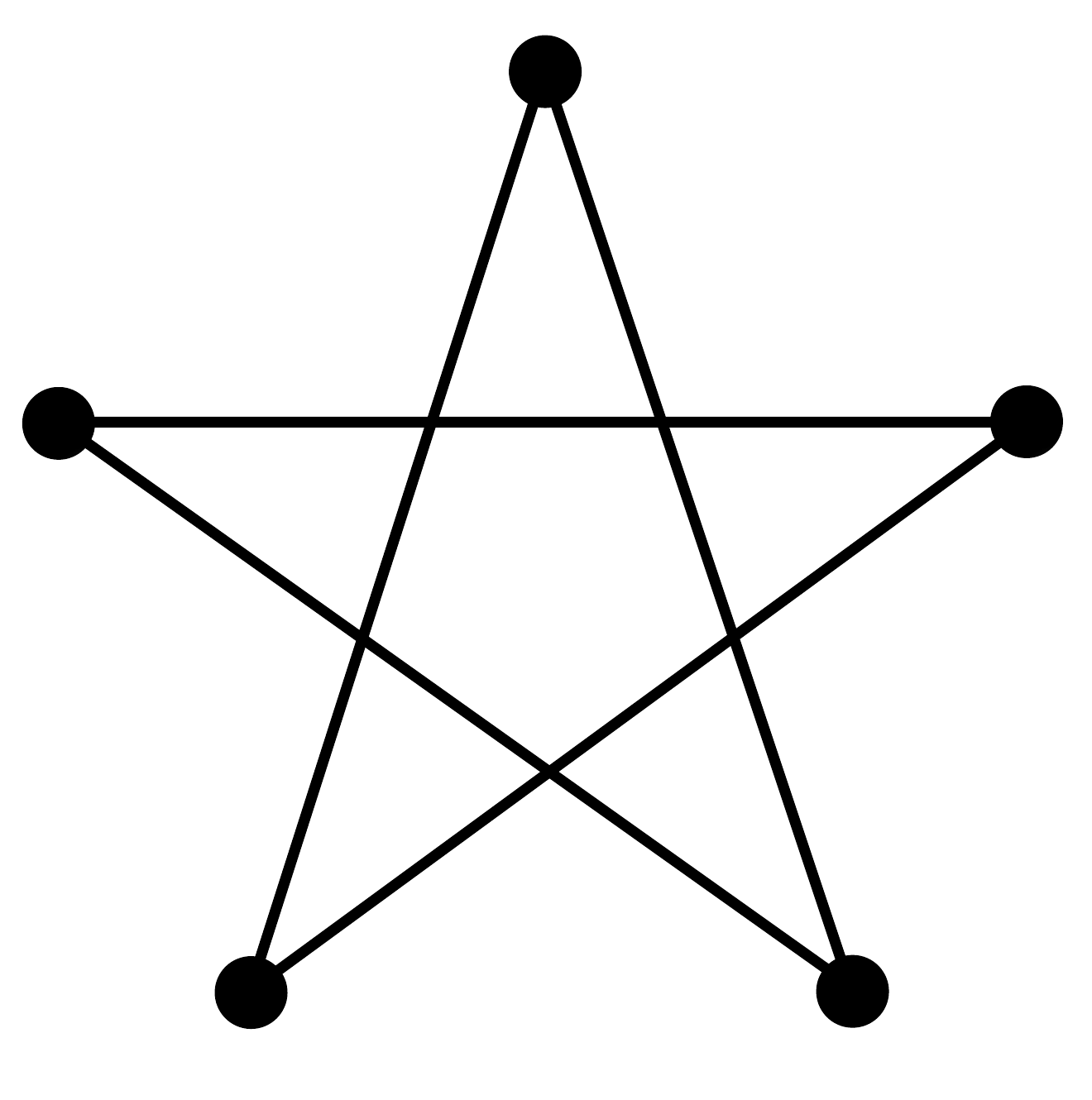}}
		&
		\subfigure[\label{subfig:cycle-7}$\overline{C}_7$]{\includegraphics[width=\ukfigwidth]{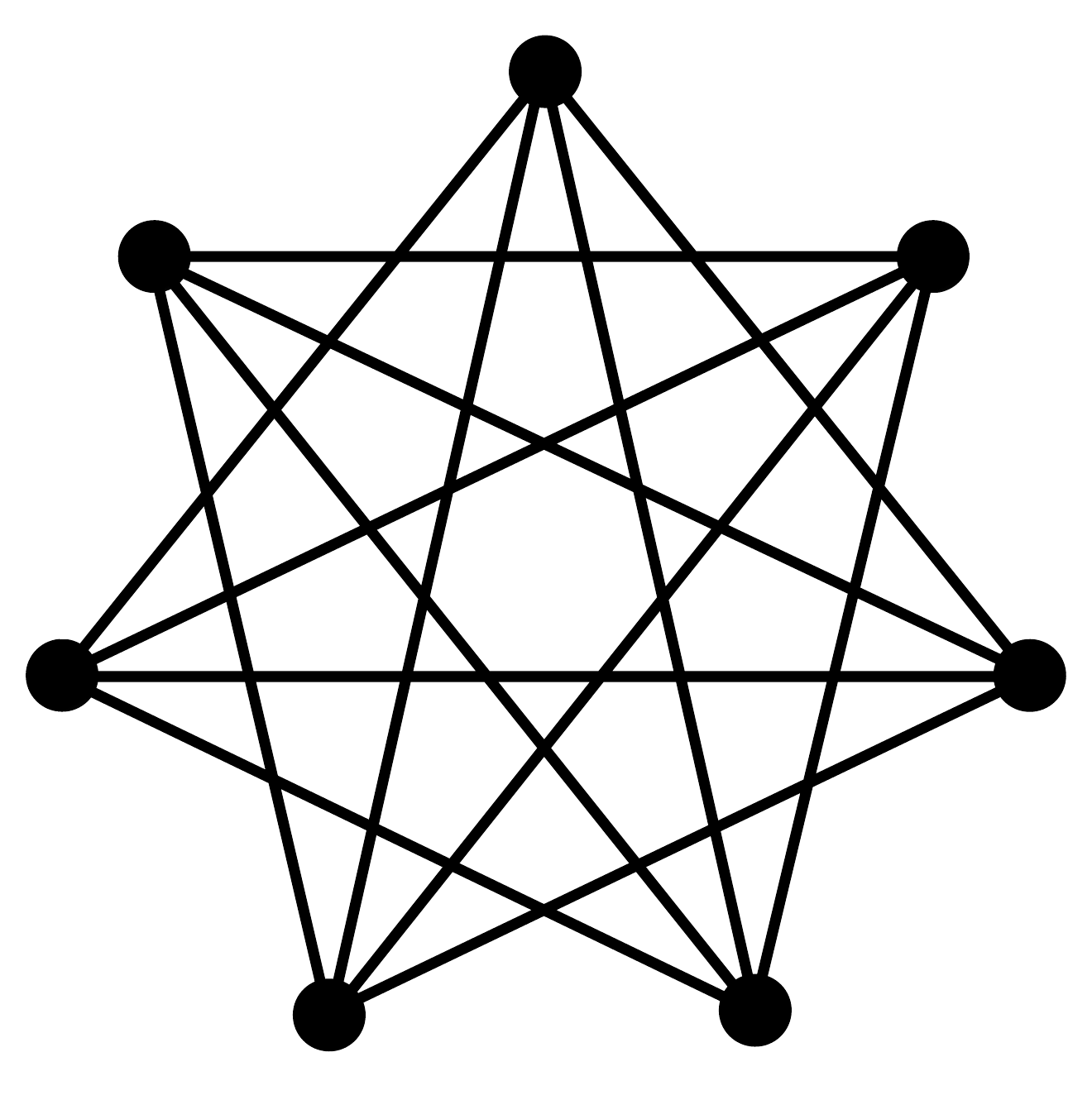}}
		&
		\subfigure[\label{subfig:cycle-9}$\overline{C}_9$]{\includegraphics[width=\ukfigwidth]{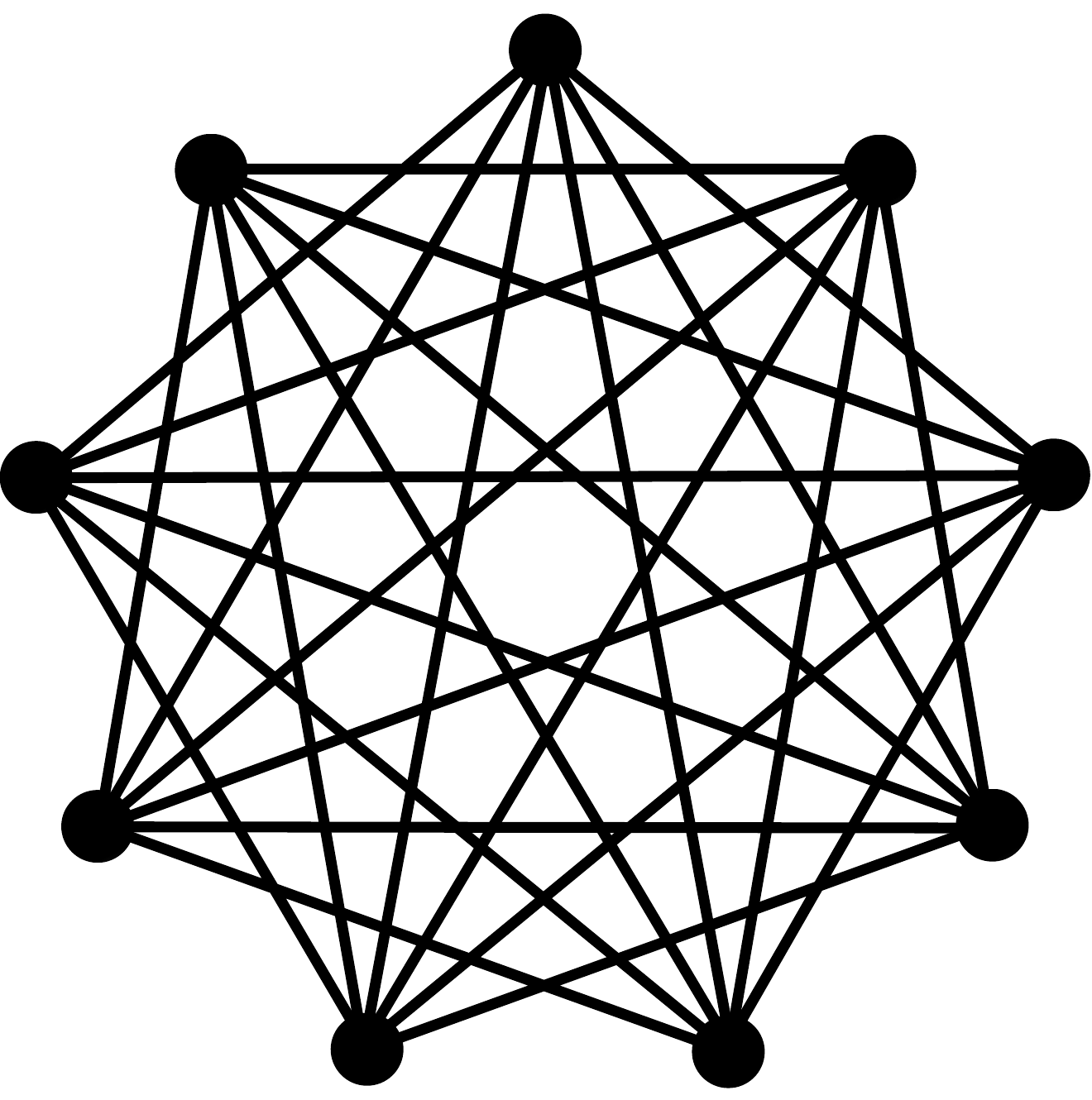}}
	\end{tabular}
	\caption{\label{fig:books}\label{fig:cycles}
		The $(r-2)$-books and complemented $(2r-1)$-cycles
		are uniquely $K_{r}$ saturated.}
\end{figure}

The complement of the $(2r-1)$-cycle is also uniquely $K_r$-saturated.
All pairs of vertices in a clique of $\overline{C}_{2r-1}$ 
	are at distance at least two 
	in the original cycle.
Such a set must have size at most $r-1$.
However, adding an edge from the cycle to its complement 
	creates a unique copy of $K_r$.
Figures \ref{subfig:cycle-5}, \ref{subfig:cycle-7},
	and \ref{subfig:cycle-9} are examples of cycle complements for $r \in \{3,4,5\}$.
Note that for $r = 3$, the complemented $(2r-1)$-cycle is isomorphic to $C_5$,
	one of the Moore graphs.

The cycle complement construction differs from the book 
	in that it gives exactly one uniquely $K_r$-saturated
	graph for each $r$.
Also of note is that the cycle complement has no dominating vertex
	(a vertex adjacent to all other vertices)
	and is regular.

Given a uniquely $K_r$-saturated graph $G$,
	adding a dominating vertex to $G$ results 
	in a uniquely $K_{r+1}$-saturated graph.
This process is reversible: given a uniquely $K_r$-saturated graph
	with a dominating vertex, remove that vertex to 
	find a uniquely $K_{r-1}$-saturated graph.
Repeating this process will eventually result in a graph with no
	dominating vertex.
Starting with the $t$-book, this process terminates in an independent set,
	which is uniquely $K_2$-saturated.
This motivates the question: 
	which uniquely $K_r$-saturated graphs
	have no dominating vertex?

\begin{conjecture}[\cite{CooperWengerCommunication}]
	For each $r$, there are a finite number of uniquely $K_r$-saturated graphs
		with no dominating vertex.
\end{conjecture}	

In an effort to generate more evidence for this conjecture,
	examples of such graphs are generated.
All known examples happen to be regular,
	which motivates the following conjecture.

\begin{conjecture}[\cite{CooperWengerCommunication}]
	For each $r$, a uniquely $K_r$-saturated graph with
		no dominating vertex is regular.
\end{conjecture}	

For $r \geq 3$, a uniquely $K_r$-saturated graph has diameter two, 
	and is 2-connected.
We apply our generation technique with an application-specific 
	pruning mechanism to 
	find these graphs.

\subsection{The Search}

In order to apply isomorph-free generation using ear augmentations,
	we must show that uniquely $K_r$-saturated graphs are 2-connected.
In fact, we prove a stronger statement using $k$-connectivity.
A graph $G$ is \emph{$k$-connected} if there exists no set $S$ of $k-1$ vertices
	so that either $G-S$ is disconnected or $G-S$ consists of a single vertex.
	
\begin{proposition}
	For all $r \geq 4$,
		if $G$ is a
		$K_r$-saturated graph
		on at least $r+1$ vertices,
		then $G$ is $(r-2)$-connected.
\end{proposition}

\begin{proof}
	If $G$ is not $(r-2)$-connected, 
		there is a set $S = \{x_1,\dots, x_{r-3}\}$
		of $r-3$ vertices
		so that $G-S$ has at 
		least two components.
	Let $u$ and $v$ be vertices in two different components.
	Then, $uv$ is not an edge in $G$.
	Since there is a copy of $K_r$ in $G+uv$,
		then there is a clique $\{y_1,y_2,\dots,y_{r-2}\}$
		of order $r-2$ so that each vertex $y_i$ in the clique
		is adjacent to both $u$ and $v$.
	At least one of the vertices $y_i$ is not in $S$, 
		so in $G-S$, 
		$u$ and $v$ are in the same component.
	This contradicts the assumption that $G-S$ is disconnected,
		and hence $G$ is $(r-2)$-connected. 
\end{proof}

Let $\cU^r$ be the class of 2-connected graphs $G$ with no copy of $K_r$ as a subgraph
	and for every edge $e\in \overline{G}$,
	there is at most one copy of $K_r$ in $G+e$.
These constraints are \emph{ear-monotone} in that
	every $G$ satisfying the constraints and any ear $\eps$
	has $G - \eps$ satisfying the constraints (except possibly 2-connectedness).
To enumerate $\cU^r$,
	we use the default canonical deletion, Delete$_{\cU}(G)$.
Since this deletion always removes a deletable ear of minimum length,
	and we are searching for uniquely $K_r$-saturated graphs
	with no dominating vertex,
	we can prune whenever our graph has a dominating vertex.
Further, since $\cU^r$ is defined by an ear-monotone property,
	we prune whenever that property is violated.

The cases for $r \in \{2,3\}$ are solved, outside of the missing Moore graph
	of degree $57$.
Hence, we run our search for $r \in \{4,5,6\}$,
	where we are guaranteed to have at least one 
	uniquely $K_r$-saturated graph
	with no dominating vertex
	of order at most $12$.
Enumerating
	$\cU_{12}^4$, $\cU_{11}^5$, and $\cU_{11}^6$
	 resulted in the following theorems.

\begin{theorem}
There are exactly three uniquely $K_4$-saturated graphs
	of order at most 12
	without a dominating vertex:
	\begin{cem}
		\item $\overline{C}_7$, on $7$ vertices of degree $4$ (Figure \ref{subfig:uk4-7}).
		\item A triangulation of the M\"obius strip, on $10$ vertices of degree $5$ (Figure \ref{subfig:uk4-10}).
		\item The icosahedron with antipodal vertices joined, 
			on $12$ vertices of degree $6$ (Figure \ref{subfig:uk4-12}).
	\end{cem}
\end{theorem}

\begin{figure}[t]
	\centering
	\mbox{
		\subfigure[\label{subfig:uk4-7}]{\includegraphics[height=1.25in]{figs-delete/uk4-7}}
		\qquad
		\subfigure[\label{subfig:uk4-10}]{\includegraphics[height=1.25in]{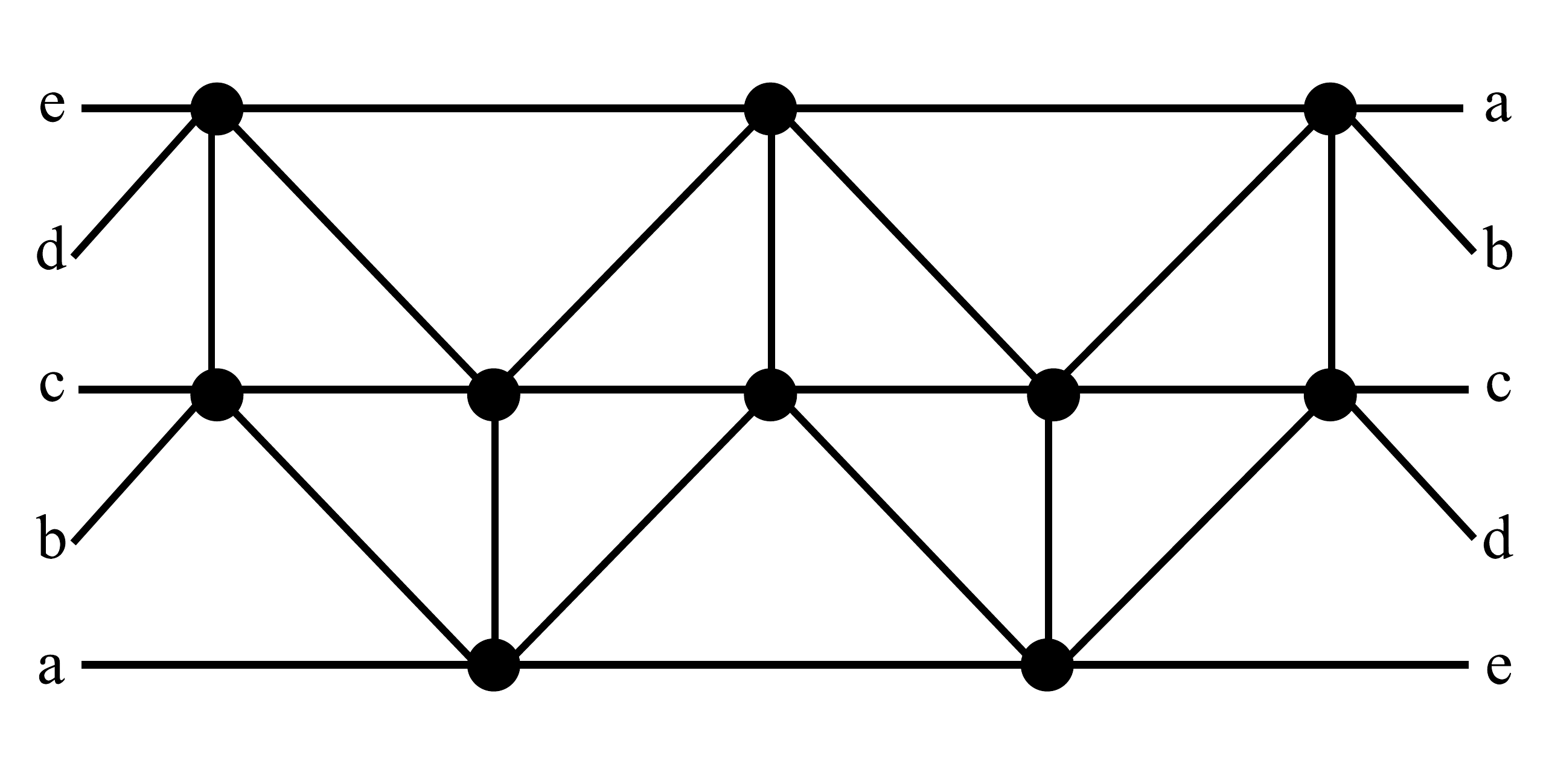}}
		\qquad
		\subfigure[\label{subfig:uk4-12}]{\includegraphics[height=1.25in]{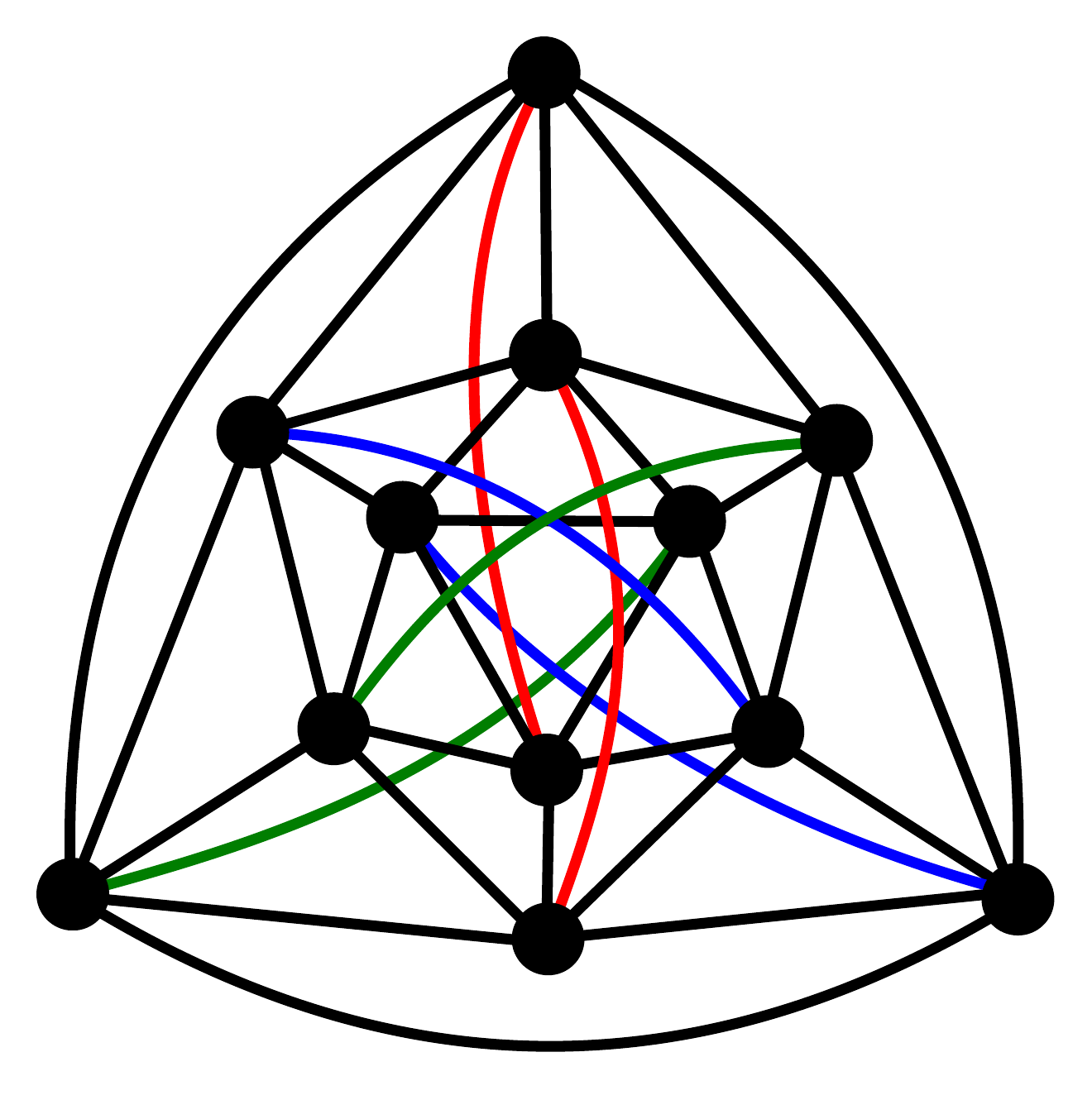}}
	}
	\caption{\label{fig:uk4}The uniquely $K_4$-saturated graphs on at most 12 vertices with no dominating vertex.}
\end{figure}

\begin{theorem}
	There is exactly one uniquely $K_5$-saturated graph
		of order at most $11$ without a dominating vertex:
		$\overline{C}_9$, on $9$ vertices of degree $6$ (Figure \ref{subfig:cycle-9}).
\end{theorem}

\begin{theorem}
	There is exactly one uniquely $K_6$-saturated graph
		of order at most $11$ without a dominating vertex:
		$\overline{C}_{11}$, on $11$ vertices of degree $8$.
\end{theorem}

While these graphs were known, 
	an exhaustive search had previously been completed
	for up to only nine vertices~\cite{CooperCommunication}.

\def\cU{\mathcal{U}}
\begin{table}[t]
	\centering
	\begin{tabular}[h]{c|r@{ }r@{ }r@{ }r@{ }r@{.}l|r@{ }r@{ }r@{ }r@{ }r@{.}l|r@{ }r@{ }r@{ }r@{ }r@{.}l}
		\multicolumn{1}{c|}{$N$} 
			 & \multicolumn{6}{c|}{CPU time for $r = 4$} 
			 & \multicolumn{6}{c|}{CPU time for $r = 5$} 
			 & \multicolumn{6}{c}{CPU time for $r = 6$}\\
		\hline 
		  8 &   
		    &   &    &   &    1&01s &
		    &   &    &   &    7&90s &
		    &   &    &   &    8&80s \\  
		  9 
		    &   &    &   &    & 31&51s  &
		     &  &  &  4m & 12&75s &
		     &  &  &  4m & 14&90s \\
		  10 & 
		  		&  &  & 29m & 31&46s   &
				&  &  5h & 24m & 38&29s &
				 &  &  8h &  0m & 47&43s \\

		  11 &
		  &   1d &  8h & 13m & 59&16s  &
		  	&  44d & 20h & 39m & 34&66s &
			 &  63d & 13h & 31m & 24&30s \\
		  12 & & 155d &  7h & 52m & 36&51s &
		      \multicolumn{6}{c|}{} &\\
	\end{tabular}
	\caption{\label{tab:UniqueK4Time}The time to search %
		for uniquely $K_r$-saturated graphs with at most $N$ vertices.}
\end{table}

The search required about $155$ days of computation time to search for uniquely $K_4$-saturated graphs on 
	up to $12$ vertices.
Timing statistics for smaller $N$ are available in Table \ref{tab:UniqueK4Time}.
Notice that as $r$ increases, the uniquely $K_r$-saturated graphs become more dense
	and the restriction on $\cU^r$ requires more graphs to be generated,
	leading to longer search times.
This caused the generation of uniquely $K_5$-saturated 
	and uniquely $K_6$-saturated graphs 
	 on twelve vertices to be left incomplete.

\section{Application 2: The Edge Reconstruction Conjecture}
\label{sec:reconstruction}

In the second application, we restrict the search to sparse 2-connected graphs 
	and utilize the structure of the search tree
	in order to minimize pairwise comparisons among the
	list of generated graphs.

\subsection{Background}

The Reconstruction Conjecture and Edge Reconstruction Conjecture 
	are two of the oldest unsolved problems in graph theory.
Given a graph $G$, the \emph{vertex deck} of $G$
	is the multiset of unlabeled graphs given by the 
	vertex-deleted subgraphs $\{ G - v : v \in V(G)\}$.
The \emph{edge deck} of $G$
	is the multiset of unlabeled graphs given by the 
	edge-deleted subgraphs $\{ G - e : e \in E(G)\}$.
A graph $G$ is \emph{reconstructible} if
	all graphs with the same vertex deck are isomorphic to $G$.
$G$ is \emph{edge reconstructible} if
	all graphs with the same edge deck are isomorphic to $G$.

\begin{conjecture}[The Reconstruction Conjecture]
	Every graph on at least three vertices is reconstructible.
\end{conjecture}

\begin{conjecture}[The Edge Reconstruction Conjecture]
	Every graph with at least four edges is edge reconstructible.
\end{conjecture}

Bondy's survey~\cite{ReconstructorsManual} 
	discusses many classic results on this topic.
Greenwell~\cite{VertexImpliesEdge} showed that the vertex deck is reconstructible from
	the edge deck, so a reconstructible graph
	is also edge reconstructible.
Therefore, the Edge Reconstruction Conjecture is weaker than
	the Reconstruction Conjecture.

Yang~\cite{Reconstruct2Connected} showed that the Reconstruction Conjecture
	can be restricted to 2-connected graphs.

\begin{theorem}[Yang~\cite{Reconstruct2Connected}]
	\label{thm:2connreconstruct}
	If all 2-connected graphs are reconstructible,
		then all graphs are reconstructible.
\end{theorem}	
	
The proof 
	considers a separable graph $G$ and tests if the
	complement $\overline{G}$ is 2-connected.
If $\overline{G}$ is 2-connected, 
	$\overline{G}$ is reconstructible (by hypothesis)
	and since the vertex deck of $\overline{G}$
	is reconstructible from the vertex deck of $G$,
	$G$ is also reconstructible.
If $\overline{G}$ is not 2-connected, Yang
	reconstructs $G$ directly using
	a number of possible cases for the structure of $G$.
There has been work to make Yang's theorem unconditional
	by reconstructing separable graphs 
	such as trees~\cite{ManvelReconstructTrees},
	cacti~\cite{ReconstructCacti, ReconstructCactiRevisited},
	and separable graphs with no vertices of degree one~\cite{ManvelSeparable},
	but separable graphs with vertices of degree one have not been proven
	to be reconstructible.

Verifying the Reconstruction Conjecture requires that 
	every pair of non-isomorphic graphs have non-isomorphic
	decks.
Running a pair-wise comparison on every pair of isomorphism classes
	on $n$ vertices
	is quickly intractable.
McKay~\cite{McKaySmallReconstruction} 
	avoided this issue and verified the conjecture
	on graphs up to $11$ vertices
	by incorporating the vertex deck
	as part of the canonical deletion.
McKay used vertex augmentations to generate the graphs, 
	so a canonical deletion
	in this search
	is essentially selecting a canonical
	vertex-deleted subgraph.
His technique selects the deletion
	based only on the vertex deck,
	so two graphs with the same
	vertex deck would be immediate siblings
	in the search tree.
With this observation,
	only siblings require pairwise comparison, 
	making the verification a reasonable computation.
We use a modification of McKay's technique 
	within the context of 2-connected graphs
	to test the Edge Reconstruction Conjecture
	on small graphs.
This strategy was first proposed in 
	unpublished work of 
	Hartke, Kolb, Nishikawa, and Stolee~\cite{HKNS09}.

\subsection{The Search Space}

To search for pairs of non-isomorphic graphs with the same edge deck,
	we adapt McKay's sibling-comparison strategy
	as well as a density argument.
If a graph has sufficiently high density,
	then the graph is edge reconstructible.

\begin{theorem}[Lov\'asz, M\"uller~\cite{LovaszDensity, MullerDensity}]
	\label{thm:reconstdensity}
	A graph on $N$ vertices and $E$ edges
		with either 
			$E > \frac{1}{2}{N\choose 2}$ or 
			$E > 1 + \log_2(N!)$
		is edge reconstructible.
\end{theorem}

Note that for all $N \geq 11$, $1 + \log_2(N!) < \frac{1}{2}{N\choose 2}$.
	
\def\cR{\mathcal{R}}
\begin{definition}
	Let $\cR_N$ be the class of 2-connected graphs $G$ with at most $N$ vertices and
		at most $1+\log_2(N!)$ edges.
\end{definition}

Note that this definition of $\cR_N$ bounds the number of edges as a function
	 of $N$ which is independent of the number of vertices of a specific graph.

\begin{corollary}
	For $N \geq 11$,
		all 2-connected graphs $G$ with at most $N$ vertices and $G \notin \cR_{N}$
		are edge reconstructible.
\end{corollary}	

We shall use $\cR_N$ as our search space.
It is deletion-closed, since removing an ear will always	decrease the number of edges.
	
Within the context of the ear-augmentation generation algorithm,
	we generate 2-connected graphs.
When trivial ears are added, these are the same as edge-augmentations.
We will show that if a non-trivial ear is added, then 
	the resulting graph is edge reconstructible 
	and its edge deck does not need to be compared to other edge decks.
Hence, an edge deck must be
	compared only when the final augmentation that 
	generated the graph is an edge augmentation,
	where the canonical deletion can be selected
	using the edge deck.
	
We begin by discussing graphs which are known to be reconstructible 
	or edge reconstructible.

\begin{proposition}
	\label{prop:detectablyreconstructable}
	A 2-connected graph  $G$ is edge reconstructible
		if any of the following hold:
	\begin{cem}
		\item There is an ear with at least two internal vertices.
		\item There is a branch vertex $v$ which is 
				incident to only non-trivial ears.
		\item $G$ is regular.
	\end{cem}
\end{proposition}

\begin{proof}
	\noindent{\bf(1)} By reconstructing the degree sequence, we
		recognize that all vertices have degree at least two.
		Since there is an ear with at least two internal vertices,
			there is an edge internal to that ear with endpoints of degree two.
		In that edge-deleted card, there are exactly two vertices
			of degree one, which must be connected by the missing edge,
			giving $G$.
			
	\noindent{\bf (2)} Let $d$ be the degree of $v$. 
			By reconstructing the vertex deck, 
			we can recognize that the card for $G-v$ 
			is missing a vertex of degree $d$ and
			that there are $d$ vertices of 
			degree one in $G-v$.
		Attaching $v$ to these vertices reconstructs $G$.
%
%

	\noindent{\bf (3)} For a $d$-regular graph $G$, 
		every edge-deleted subgraph $G-e$ has
		exactly two vertices of degree $d-1$
		corresponding to the endpoints of $e$. 
\end{proof}

	Graphs satisfying any of the conditions of 
		Proposition \ref{prop:detectablyreconstructable}
		are called \emph{detectably edge reconstructible} graphs.

\subsection{Canonical deletion in $\cR_N$}

In this section, we describe a method for selecting a canonical ear
	to delete from a graph in $\cR_N$.


If we are able to determine that $G$
	is edge reconstructible,
	then the canonical deletion does not need to be generated from the 
	edge deck.
In such a case, 
	we default to the canonical deletion algorithm Delete$_{\cF}(G)$,
	where the canonical labeling of $G$ gives the lex-first 
	ear $\eps$ of minimum length so that $G-\eps$ 2-connected.

If $G$ is not detectably edge reconstructible,
	then all ears of $G$ have at most one internal vertex,
	and every branch vertex is incident to at least one trivial ear.
These properties allow us to find either
	a trivial ear 
	or an ear of order one 
	whose deletion remains 2-connected.
Compute the minimum $r$
	so that there exists an ear $\eps$ in $G$ of order $r$
	so that $G-\eps$ is 2-connected.
We prefer to select a trivial ear  
	when available.
	

Out of the choices of possible order-$r$ ear deletions,
	count the multiplicities for the degree set of the ear endpoints.
Find the pair $\{d_1,d_2\}$ of endpoint degrees
	which has minimum multiplicity over all deletable ears of order $r$ in $G$
	breaking ties by using the lexicographic order.
Out of the deletable ears of order $r$ and endpoint degrees $\{d_1,d_2\}$,
	we must select a canonical ear using the edge deck.
If $r = 0$, any trivial deletable ear $\eps$ corresponds 
	to the edge-deleted subgraph
	$G-\eps$.
By computing the canonical labels of these cards and selecting 
	the lexicographically-least canonical string,
	we can select a canonical edge.
If $r = 1$, there are two edges in the ear that can be deleted to form
	edge-deleted subgraphs with a single vertex of degree 1 connected
	to a 2-connected graph.
We compute the canonical labels of both cards, 
	select  
	the lexicographically-least canonical string,
	then find the lex-least string of those strings.
	
Due to the nature of the reconstruction problem,
	this canonical deletion procedure is not perfect.
There are graphs $G$ containing trivial ears $\eps_1, \eps_2$
	whose deletions $G-\eps_1$ and $G-\eps_2$ are isomorphic, 
	but $\eps_1$ and $\eps_2$ are not in orbit within $G$.
If the edge-deleted subgraph $G-\eps_1$ is selected
	as the canonical edge card,
	the deletion algorithm
	must accept both $\eps_1$ and $\eps_2$
	as canonical deletions.
This leads to a duplication of $G$ in the search tree,
	but only in the limited case of a graph $G$
	which is not detectably edge reconstructible
	\emph{and} such ambiguity appears.
A similar concern occurs for the vertex-deletion case, 
	but is not explained in~\cite{McKaySmallReconstruction}.

To compare graphs with the same canonical deletion,
	we use three comparisons.
The first compares the degree sequences.
The second compares a custom reconstructible invariant\footnote{
	This invariant is not theoretically interesting,
		but is available
		in the source code. See 
		the \texttt{GraphData::computeInvariant()} method.
	},
	which is based on the degree sequence
	of the neighborhood of each vertex.
The third and final check compares the 
	sorted list of canonical strings for
	the edge-deleted subgraphs.
During the search,
	there was no pair of graphs which satisfied all
	three of these checks.

\subsection{Results}

With the canonical deletion Delete$_{\cR}(G)$, 
	$\cR_N$
	was generated and checked for collisions
	in the edge decks of graphs which are not 
	detectably reconstructible.
Table \ref{tab:ReconstructionTime} describes the computation time
	for $N \in \{8,\dots,12\}$.

\begin{table}[t]
	\centering
	\begin{tabular}[h]{c|r|r|r|r|r|r@{ }r@{ }r@{ }r@{ }r@{.}l}
		\multicolumn{1}{c|}{$N$} 
			&\multicolumn{1}{c|}{$g(N)$}
			 & \multicolumn{1}{c|}{$|\cR_N|$} 
			 & \multicolumn{1}{c|}{Diff 1}
			 & \multicolumn{1}{c|}{Diff 2}
			 & \multicolumn{1}{c|}{Diff 3}
			 & \multicolumn{6}{c}{CPU time}\\
		\hline 
		  8 & 16 
		  	&        4804  
			& 145 & 177 & 187 & &&&& 8&01s \\
		  9 & 19 
		  	&      111255 
		  	& $6.19 \times 10^3$ & $5.72 \times 10^3$ & $4.77 \times 10^3$
			 & &  &  &  5m & 33&85s \\
		 10 & 22 
		 	&     3051859 
		 	& $7.13 \times 10^5$ & $6.00 \times 10^5$ & $4.21 \times 10^5$ 
			 & &  &  6h & 33m & 40&59s \\
		 11 & 26 
		 	&   308400777 
		 	& $9.44 \times 10^7$ & $7.28 \times 10^7$ & $3.83 \times 10^7$ 
			& &  32d & 20h & 38m & 08&16s \\
		 12 & 29 
		 	& 25615152888
		 	& $12.00 \times 10^{9}$ & $9.60 \times 10^9$ & $4.47 \times 10^9$ 
			& 10y & 362d & 13h & 05m & 39&13s \\
	\end{tabular}
	\caption{\label{tab:ReconstructionTime}Comparing $|\cR_N|$ and %
		the time to check $\cR_{N}$. Here, $g(N) = 1+\lfloor\log_2(N!)\rfloor$.}
\end{table}

With this computation, we have the following theorem.

\begin{theorem}
	\label{thm:2connreconst12}
	All 2-connected graphs on at most $12$ vertices 
		are edge reconstructible.
\end{theorem}

This computation extends the previous result that 
	all graphs of order at most $11$ are vertex reconstructible~\cite{McKaySmallReconstruction}.
To remove the 2-connected condition of Theorem \ref{thm:2connreconst12},
	there are three possible methods.
First, prove Yang's Theorem (Theorem \ref{thm:2connreconstruct})
	for the edge reconstruction problem.
Second, Yang's Theorem could be 
	made unconditional by 
	proving that separable graphs 
	are reconstructible or edge reconstructible.
Third, a second stage of search could be designed to combine a 
	list of two-connected graphs to form sparse separable graphs
	and test edge reconstruction on those cases.

\arbitrarypagebreak
\section{Conclusion}

Generating 2-connected graphs by ear augmentations
	and removing isomorphs by canonical ear deletion
	is an effective and general technique.
The computation times show that the technique
 	is more effective in the case of ear-monotone properties 
	such as uniquely $K_4$-saturated graphs
	or when the structure of the ear decomposition
	is essential to the problem at hand,
	such as verifying the edge reconstruction conjecture on small graphs.

A forthcoming work~\cite{pExtremal} applies the generation technique
	to search for dense graphs with a fixed number of perfect matchings
	(see~\cite{DudekSchmitt} and~\cite{HSWY} for background).
Previous work~\cite{HSWY} 
	classified the infinite family of graphs
	into a particular combination of finite pieces,
	which can be found through our generation process.
This results in exact structure theorems
	for a larger class of parameters, where the exact structure
	is computationally generated.
Our implementation is general enough 
	to allow for such extensions to generate
	other families of 2-connected graphs,
	and is concurrent to allow for large computations to be 
	run quickly in real time.

\section*{Acknowledgements}

The author thanks 
	Stephen G. Hartke,
	Hannah Kolb, 
	Jared Nishikawa, 
	Kathryn T. Stolee,
	and Paul Wenger
	for interesting discussions 
	concerning the problems addressed in this work
	and for improving the quality of this work.
Special thanks to Joshua Cooper for Figure \ref{subfig:uk4-10}.

This work was completed utilizing the Holland Computing Center of the University of Nebraska.
Thanks to the Holland Computing Center faculty and staff
	including Brian Bockelman, Derek Weitzel, and David Swanson.
Thanks to the Open Science Grid for access to significant computer resources.
The Open Science Grid is supported by the National Science Foundation 
	and the U.S. Department of Energy's Office of Science.


\bibliographystyle{abbrv}
\bibliography{IsomorphFreeGeneration}

\end{document}